\documentclass{amsart}
\usepackage[utf8x,utf8]{inputenc}
\usepackage{amsfonts,amsmath,amssymb,amsthm}
\usepackage{amscd}
\usepackage{mathrsfs}
\usepackage{ucs,bbm} 
\usepackage{colonequals}
\usepackage[all]{xy}
\usepackage{verbatim}
\usepackage[backref=page]{hyperref}

\hypersetup{
     colorlinks   = true,
     citecolor    = red,
     linkcolor = blue,
     urlcolor = purple
}

\usepackage{enumerate}
\usepackage{xcolor}

\mathchardef\dash="2D

\newcommand{\dR}{\mathrm{dR}}

\newcommand{\Q}{\mathbb{Q}}
\newcommand{\QQ}{\mathbb{Q}}

\newcommand{\Z}{\mathbb{Z}}

\newcommand{\FOg}{\mathbf{FOg}}% filtered ogus
\newcommand{\DM}{\mathbf{DM}} % voevodsky motives
\newcommand{\Og}{\mathbf{Og}}% ogus
\DeclareMathOperator{\iHom}{\underline{\mathrm{Hom}}}

\DeclareMathOperator{\Cone}{Cone}

\DeclareMathOperator{\Hom}{Hom}

\DeclareMathOperator{\id}{id}

\DeclareMathOperator{\D}{D}

\DeclareMathOperator{\Ext}{Ext}

\DeclareMathOperator{\Coker}{Coker}

\DeclareMathOperator{\coker}{coker}
\DeclareMathOperator*{\colim}{colim}

\usepackage{fancyhdr}

\title{Extensions of filtered Ogus structures}
\author{Bruno Chiarellotto}
\address{Universit\`a di Padova, Dipartimento di Matematica ``Tullio Levi-Civita''}
\email{chiarbru@math.unipd.it}
\author{Nicola Mazzari}
\address{Université de Bordeaux, CNRS, Bordeaux INP, IMB, UMR 5251}
\email{nicola.mazzari@math.u-bordeaux.fr}
\newtheorem{exo}{Exercise}[section]

\newtheorem{lmm}[exo]{Lemma}
\newtheorem{prp}[exo]{Proposition}

\theoremstyle{definition}
\theoremstyle{remark}\newtheorem{rmk}[exo]{Remark}
\theoremstyle{remark}
\theoremstyle{definition}

\begin{document}

%\pagestyle{fancy}
%\lhead{}
%\rhead{}

\maketitle
\begin{abstract}
    We compute the Ext group of the 
	  (filtered) Ogus category over a number field $K$. In particular we prove that the filtered Ogus realisation of mixed motives is not fully faithful.
\end{abstract}

%\footnotesize{\tableofcontents}
\section{Introduction}
Recently Andreatta, Barbieri-Viale and Bertapelle \cite{ABVB16} have defined the filtered Ogus realisation $T_{\FOg}$ for $1$-motives over a number field $K$. In fact by \cite{ChiLazMaz:19a} there exists a cohomology theory  for $K$-varieties with values in $\FOg(K)$ compatible with $T_{\FOg}$. More precisely let $\mathbf{DM}_{gm}(K)$ be the Voevodsky's category of geometric motives over $K$, then there exists a (homological) realisation functor
   	\[
   		R_{\FOg}:\DM_{gm}(K)\to \D^b(\FOg(K))\ 
   	\]
	compatible with $T_{\FOg}$.

The aim of this paper is to compute the Ext group in $\FOg$ (Proposition~\ref{prp:extformula}). We follow the method of Beilinson \cite{Beu:86a,Ban:02a}. 

It follows (see remark~\ref{rmk:badext}) that the filtered Ogus realisation of mixed motives is not fully faithful in general, even though  $T_{\FOg}$ is fully faithful.

\subsection{Notations and conventions}

Throughout this article, $K$ will denote a number field. A place of $K$ will always mean a finite place (we will never need to consider real or complex places). For every such place $v$ of $K$, let $K_v$ denote the completion, $\mathcal{O}_v$ the ring of integers, $k_v$ the residue field, $p_v$ its characteristic, and $q_v=p_v^{n_v}$ its order. For all $v$ which are unramified over $\Q$, let $\sigma_v$ denote the lift to $K_v$ of the absolute Frobenius of $k_v$. 
\section{The categories}
\subsection{The Ogus category}\label{sec:fog}
 Let $P$ be a cofinite set of absolutely unramified places of $K$. We define $\mathcal{C}_P$ to be the category whose objects are systems $M=(M_\dR,(M_v,\phi_v,\epsilon_v)_{v\in P})$ such that:
\begin{enumerate}
	\item $M_\dR$ is a finite dimensional $K$-vector space;
	\item $(M_v,\phi_v)$ is a $F$-$K_v$-isocrystal, that is, $M_v$ is equipped with a $\sigma_v$-linear automorphism $\phi_v$;
	\item $\epsilon=(\epsilon_v)_{v\in P}$ is a system of $K_v$-linear isomorphisms
	\[
		\epsilon_v:M_\dR\otimes K_v\to M_v\ .
	\]
\end{enumerate}

A morphism $f:M\to M'$ is then a collection $(f_\dR,(f_v)_{v\in P})$ where:
\begin{enumerate}
	\item $f_\dR:M_\dR\to M_\dR'$ is a $K$-linear map;
	\item $f_v:M_v\to M_v'$ is $K_v$-linear morphism compatible with Frobenius and such that $\epsilon_v^{-1} \circ f_v \circ \epsilon_v=f_\dR\otimes K_v$.
\end{enumerate}
Note that by the second criterion, to specify a morphism it is enough to specify $f_\dR$. There are obvious `forgetful' functors $\mathcal{C}_P\rightarrow \mathcal{C}_{P'}$ whenever $P'\subset P$ and we can form the Ogus category $\Og(K)$  as the 2-colimit
\[
	\Og(K)=2\colim_{P} \mathcal{C}_P
\]
where $P$ varies over all cofinite sets of unramified places of $K$. For an object $M\in \Og(K)$ and $n\in \Z$ we denote by $M(n)$ the Tate twist of $M$, that is where each Frobenius $\phi_v$ is multiplied by $p_v^{-n}$.

\subsection{Weights}
 A \emph{weight filtration} on an object $M=(M_\dR,(M_v,\phi_v,\epsilon_v)_{v\in P})\in \mathcal{C}_P$ is an increasing filtration $W_\bullet M$ by sub-objects in $\mathcal{C}_P$ such that for all $v\in P$ the graded pieces $\mathrm{Gr}_i^WM_v$ are pure of weight $i$. That is, all eigenvalues of the linear map $\phi_v^{n_v}$ are Weil numbers of $q_v$-weight $i$ (i.e. all their conjugates have absolute value $q_v^{i/2}$ \cite{Chi98}). Again, to give a weight filtration on $M$ it suffices to give a filtration on $M_\dR$ which induces a weight filtration on all $M_v$.

\subsection{The filtered Ogus category}
We can therefore consider the filtered Ogus category $\FOg(K)$ whose objects are objects of $\Og(K)$ equipped with a weight filtration, and morphisms are required to be compatible with this filtration.

\begin{lmm}[\cite{ABVB16}, Lemma 1.3.2] The filtered Ogus category $\FOg(K)$ is a $\Q$-linear abelian category, and the forgetful functor
\[ \FOg(K)\rightarrow \Og(K) \]
is fully faithful. 
\end{lmm} 
\subsection{Internal Hom }
If $M,N$ are two objects in $\FOg$ then we can define the internal Hom, denoted by $\iHom_\FOg(M,N)$ as follows:
\begin{enumerate}
	\item $\iHom_\FOg(M,N)_\dR:=\iHom_K(M_\dR,N_\dR)$ is just the usual Hom of $K$-vector spaces.
	\item for  all $v$, $\iHom_\FOg(M,N)_v:=\iHom_{K_v}(M_v,N_v)$ and for almost all $v$ this $K_v$-vector space is endowed with the Frobenius $f\mapsto \phi_v^N\circ f\circ (\phi_v^M)^{-1}$
	\item $W_r\iHom_\FOg(M,N):=\{f\in \iHom_\FOg(M,N):f(W_iM)\subset W_{i+r}N\}$.
\end{enumerate}

\section{Ext computation}
Let $M,N$ be two objects in $C^b(\FOg)$ (the category of bounded complexes of $\FOg$) and consider the following complexes
\begin{align*}
	A(M,N)&=W_0\iHom^\bullet(M,N)_\dR
	\\
	&=W_0\iHom_K^\bullet(M_\dR,N_\dR)\\
	B(M,N)
	&=
	\prod_v'W_0\iHom^\bullet(M,N)_v	\quad (\text{restricted product})
\end{align*}
and the morphism
\[
	\xi_{M,N}:A(M,N)\to B(M,N)\qquad \xi(x)=(x\phi_M-\phi_Nx),\ .
\]
We want to prove that the cone of this map compute the ext-groups of $\FOg$, i.e.
\[
	\Ext^i_{\FOg}(M,N)\cong H^{i-1}(\Cone(\xi_{M,N}))\ .
\]
\begin{lmm}
	Let $\xi_{M,N}$ as above, then for any $i$ and for any element $b\in B^i(M,N)$  there exist a quasi-isomorphism $N\to E$ of complexes such that the image of $b$ in  $\Coker (\xi_{M,E})$ is zero.
\end{lmm}
\begin{proof}
	%(We follow Bannai, looks like there is a mistake in the proof of 1.8 but that of 2.15 is correct) 
	Take $b\in B^0(M,N)$, so that $b=(b^i)$ with $b^i\in \prod_v'W_0\iHom(M^i,N^i)_v$. Then we construct $E$ as follows
	\[
		E:=\Cone ((0,\id): M[-1]\to N\oplus M[-1])
	\]
	where everything is defined  as expected but the Frobenius: $\phi_E$ on $E^i=N^i\oplus M^{i-1}\oplus M^i$ is given by
	\begin{align*}
		\phi_E(x,0,0)=&(\phi_N(x),0,0)\\
		\phi_E(0,y,0)=&(b^id_My-d_Nb^{i-1},\phi_M(y),0)\\
		\phi_E(0,0,z)=&(-b^iz,0,\phi_Mz)
	\end{align*}
	%(check there are no problems in this definition)
	
	By construction $N\to E$ is a quasi-isomorphism and there is a short exact sequence
	\[
		0\to N\to E\to \Cone(\id_M)[-1]\to 0 \ .
	\]

	Finally we remark that the natural map $B^0(M,N)\to B^0(M,E)$ sends $b$ to $(b,0,0)$ and we can explicitly compute
	\[
		\xi_{M,E}(0,0,\id)=(0,0,\id)\phi_M-\phi_E(0,0,\id)=(0,0,\phi_M)-(-b,0,\phi_M)=(b,0,0)
	\]
	as expected.
\end{proof}
\begin{prp}\label{prp:extformula}Let $M,N$ be two complexes in $C^b(\FOg)$
	\[
		\Ext^i_{\FOg}(M,N)\cong H^{i-1}(\Cone(\xi_{M,N}))\ .
	\]
\end{prp}
\begin{proof}
	The proof is similar to \cite[Proposition 1.7]{Ban:02a}.
	We have by definition
	\[
		\Ext^i_{\FOg}(M,N)=\Hom_{D^b(\FOg)(M,N[i])}=\colim_I\Hom_{K^b(\FOg)}(M,L[i])
	\]
	where $I$ is the category of quasi-isomorphisms $s:N\to L$ in the homotopy category $K^b(\FOg)$.
	
	By the octahedron axiom and the exactness of $A(M,-), B(M,-)$ there is a long exact sequence
	\[
		H^i(\ker \xi_{M,N})\to H^i(\Cone(\xi_{M,N})[-1])\to H^i(\coker(\xi_{M,N})[-1])\to +\ .
	\]
	Note that $H^i(\ker \xi_{M,N})=\Hom_{K^b(\FOg)}(M,N[i])$. By the previous lemma
	\[
		\colim_I H^i(\coker(\xi_{M,L})[-1])= 0\ .
	\]
	Thus we obtain the expected result by taking  the colimit over $I$ of the above long exact sequence.
	
	 We can also give a direct proof in the case of chain complexes concentrated in degree zero, as explained in the following remark.
\end{proof}
\begin{rmk}
	When $M,N\in \FOg$ we can derive the above formula as follows. Let 
	\[
		0\to N\to E\xrightarrow{\pi} M\to 0
	\]
	be an extension in $\FOg$. Choose a section $s_\dR\in W_0\iHom(M,E)_\dR$ of $\pi_dR$. After base change to $K_v$ we get sections $s_v\in W_0\iHom(M,E)_v$ and we can define (for almost all $v$) 
	\[
		x_v:=s_{v}\circ\phi_{M_v}-\phi_{E_v}\circ s_v\ .
	\]
	It follows that $x_v\in W_0\iHom(M,N)_v$ so that $x=(x_v)_{v}$ is an element of $\prod_v' W_{0}\iHom(M,N)_v$. Starting with another section $s_\dR'$ we will get another $x'$ and the difference $x-x'$ lies in $(\circ\phi_M-\phi_N\circ)W_0\iHom(M,N)_\dR$ by construction. Then we easily get a map 
	\[
		\Phi:\Ext^1_{\FOg}(M,N)\to \frac{\prod_v' W_{0}\iHom(M,N)_v}{(\circ\phi_M-\phi_N\circ)W_0\iHom(M,N)_\dR},\quad \Phi(E)=(x_v)_v
	\]
	Moreover given a family $x=(x_v)_v$ as above we can  define the extension $E_x$ to be the direct sum $N\oplus M$ except for the fact that we set the Frobenius to be
	\[
		\phi_{E,v}(n,m):=(\phi_{N,v}(n)-x_v(m),\phi_{M,v}(m))\ .
	\]
	By construction we have $\Phi(E_x)=x$ and we prove that $\Phi$ is an isomorphism.
\end{rmk}
%\begin{rmk}
%	Let $\FOg_i\subset \FOg$ be the full subcategory of poure of weight $i$ objects. If $M,N\in \FOg_i$, then 
%	\[
%		\Ext^1_\FOg(M,N)=\frac{\prod_v'\iHom(M,N)_v}{(\phi_M-\phi_N}\iHom(M,N)_\dR
%	\]
%	which is trivial only if $\Hom(M,N)$ is trivial.
%\end{rmk}
%
\begin{prp}
	Let $M,N\in \FOg$ there is a short  exact sequence 
	\[
		0\to \Ext^1_{\FOg}(M,N)\to \Ext^1_{\Og}(M,N)\to \frac{\prod_v' W_{\ge 1}\iHom(M,N)_v}{(\circ\phi_M-\phi_N\circ)\iHom(M,N)_\dR}\to 0 \ .
	\]
\end{prp}
\begin{proof}
	The methods we have introduced to compute the extension groups in $\FOg_K$ work also for $\Og_K$. In fact   we consider the above construction forgetting about weights
	\begin{align*}
		A'(M,N)&=\iHom^\bullet(M,N)_\dR\\
		B'(M,N)
		&=
		\prod_v'\iHom^\bullet(M,N)_v	\\
		\xi_{M,N}':&A'(M,N)\to B'(M,N)\qquad \xi(x)=(x\phi_M-\phi_Nx)
		\end{align*}
		so that we have
		\[
			\Ext^i_{\Og}(M,N)\cong H^{i-1}(\Cone(\xi_{M,N}')) \ .
		\]
		In particular if $M,N\in \FOg$ there is an exact sequence of complexes
		\[
			0\to \Cone(\xi_{M,N})=W_0\Cone(\xi_{M,N}')\to \Cone(\xi_{M,N}')\to W_{\ge 1}\Cone(\xi_{M,N}')\to 0
		\]
		whose associated long exact sequence degenerates to the short exact sequence of the statement.
\end{proof}
\begin{rmk}\label{rmk:thick}
	Let us consider an intermediate category $\FOg\subset \FOg'\subset \Og$ whose objects are $M\in \Og(K)$ endowed with an increasing  filtration $M_i\subset M_{i+1}$ (without any condition on Frobenius eigenvalues). This is just an exact category and it is not full in $\Og$. Nevertheless $\FOg\subset\Og$ is full and for two objects $M,N\in \FOg$ we have
	\[
		\Ext^1_{\FOg}(M,N)\cong \Ext^1_{\FOg'}(M,N)
	\] 
	just following the previous proof.
\end{rmk}
\begin{rmk}\label{rmk:badext}
	It follows from the previous proposition that the $\Ext^1_{\FOg}(M,N)$ are not countable in general and in particular different from motivic cohomology. For instance already for $K=\QQ$  we get
	\[
		\Ext^1_{\FOg}(\QQ,\QQ(1))\cong\frac{\{(a_p)_p\in \prod_p'\QQ_p\}}{\{(b-p^{-1}b)_p:\ b\in \QQ\}}\ ,
	\]
	which is uncountable and so different from $\Ext^1_{\DM}(\QQ,\QQ(1))=\QQ^*\otimes \QQ$. Hence the filtered Ogus realisation of mixed motives in not full.
\end{rmk}

%\bibliographystyle{plain}%bibsty
%\bibliography{2018_ogus_lib.bib}

\end{document}